\documentclass[]{article}

\usepackage[leqno]{amsmath}
\usepackage{amssymb,amsfonts,xfrac,MnSymbol}
\usepackage{amsthm}
\usepackage{cite}
\usepackage{hyperref}
\usepackage{todonotes}
\usepackage{enumitem}
\usepackage{graphicx}
\usepackage{tikz-cd, tikz}
\usepackage{color}
\usepackage{import}
\usepackage{cleveref}
\usepackage{yfonts}
\usepackage{stmaryrd}
\usepackage{commath}
\usepackage{calc}

\makeatletter
\def\namedlabel#1#2{\begingroup
    #2%
    \def\@currentlabel{#2}%
    \phantomsection\label{#1}\endgroup
    \newcounter{tmp@cnt}
\newcommand*\@labelpunc{.}
}
\makeatother

\numberwithin{equation}{section}

\newtheorem{theorem}{Theorem}[section]
\newtheorem{claim}[theorem]{Claim}
\newtheorem{proposition}[theorem]{Proposition}
\newtheorem{lemma}[theorem]{Lemma}
\newtheorem{corollary}[theorem]{Corollary}

\newtheorem*{theorem*}{Theorem}
\newtheorem*{claim*}{Claim}
\newtheorem*{proposition*}{Proposition}
\newtheorem*{lemma*}{Lemma}
\newtheorem*{corollary*}{Corollary}

\newtheoremstyle{named}{}{}{\itshape}{}{\bfseries}{.}{.5em}{\thmnote{#3}}
\theoremstyle{named}
\newtheorem*{namedtheorem}{Theorem}

\theoremstyle{definition}
\newtheorem{definition}[theorem]{Definition}

\newtheorem{remark}[theorem]{Remark}

\newtheorem{notation}[theorem]{Notation}

\newtheorem*{definition*}{Definition}
\newtheorem*{observation*}{Observation}
\newtheorem*{remark*}{Remark}
\newtheorem*{example*}{Example}
\newtheorem*{question*}{Question}
\newtheorem*{exercise*}{Exercise}
\newtheorem*{fact*}{Fact}
\newtheorem*{notation*}{Notation}

\newcommand{\bbH}{\mathbb{H}}

\newcommand{\bbN}{\mathbb{N}}

\newcommand{\bbQ}{\mathbb{Q}}
\newcommand{\bbR}{\mathbb{R}}

\newcommand{\bbZ}{\mathbb{Z}}

\newcommand{\bfc}{\mathbf{c}}

\newcommand{\bfs}{\mathbf{s}}
\newcommand{\bft}{\mathbf{t}}

\newcommand{\bfw}{\mathbf{w}}

\newcommand{\calA}{\mathcal{A}}
\newcommand{\calB}{\mathcal{B}}
\newcommand{\calC}{\mathcal{C}}
\newcommand{\calD}{\mathcal{D}}

\newcommand{\calF}{\mathcal{F}}

\newcommand{\calH}{\mathcal{H}}

\newcommand{\calL}{\mathcal{L}}

\newcommand{\calN}{\mathcal{N}}

\newcommand{\calP}{\mathcal{P}}

\newcommand{\actson}{\curvearrowright}

\newcommand{\bdelta}{\boldsymbol{\delta}}

\newcommand{\tild}[1]{\widetilde{#1}}

\newcommand{\comb}[1]{\calC_{comb}(#1)}
\newcommand{\cyl}[1]{\calC_{cyl}(#1)}
\newcommand{\cusped}[1]{\calC(#1)}
\newcommand{\geod}[1]{\llbracket {#1} \rrbracket}
\newcommand{\checkH}[2]{{\check{H}^{#1}}{(#2)}}
\newcommand{\vol}[1]{\textsc{Vol}(#1)}
\newcommand{\immers}{\looparrowright}

\DeclareMathOperator{\df}{def}
\DeclareMathOperator{\supp}{supp}

\DeclareMathOperator{\cd}{cd}
\DeclareMathOperator{\C}{C}
\DeclareMathOperator{\depth}{Depth}
\DeclareMathOperator{\prd}{pd}
\DeclareMathOperator{\Isom}{Isom}

\title{Finite Index Rigidity of Relatively Hyperbolic Groups}
\author{Nir Lazarovich\thanks{Supported by the Israel Science Foundation (grant no. 1576/23)}, Gon Rahamim, and Alessandro Sisto}
\date{September 2025}

\newcommand{\Addresses}{{% additional braces for segregating \footnotesize
  \bigskip
  \footnotesize

  \bigskip
    \noindent
  \textsc{Department of Mathematics, Technion -- Israel Institute of Technology, Haifa, Israel}\par\nopagebreak \noindent
    \textit{E-mail address:} \texttt{lazarovich@technion.ac.il}

   \bigskip
    \noindent
  \textsc{Department of Mathematics, Technion -- Israel Institute of Technology, Haifa, Israel}\par\nopagebreak \noindent
    \textit{E-mail address:} \texttt{gon.rahamim@campus.technion.ac.il}

    \bigskip
    \noindent
  \textsc{Department of Mathematics, Heriot-Watt University and Maxwell Institute for Mathematical Sciences, Edinburgh, UK}\par\nopagebreak \noindent
    \textit{E-mail address:} \texttt{a.sisto@hw.ac.uk}

}}

\begin{document}

\maketitle{}
\begin{abstract}
We prove that, given a torsion-free relatively hyperbolic group $G$ with non-relatively-hyperbolic peripherals, isomorphic finite index subgroups of $G$ have the same index. This applies for instance to fundamental groups of finite-volume negatively curved manifolds, to limit groups, and to free-by-cyclic groups. More generally, we show that if two finite index subgroups of a relatively hyperbolic group are isomorphic via a map that respects their peripheral structures, then their indices in the ambient group are equal. The proof relies on demonstrating that the number of simplices in a simplicial classifying space of a finite index subgroup in a relatively hyperbolic group grows linearly with its index. These results generalize earlier work of the first author in the context of hyperbolic groups.

% We prove that isomorphic finite index subgroups of toral relatively hyperbolic groups have the same index. More generally we show that if two finite index subgroup of a relatively hyperbolic group admit an isomorphism that respects their peripheral structure, then they have the same index.
% The argument proceeds by showing that the complexity of a finite index subgroup of a relatively hyperbolic group is linear in its index. 
% These results extend previous results by the first author for hyperbolic groups.
% Lazarovich \cite{lazarovich2023finite} proved that any two finite index subgroups of a non-elementary hyperbolic group have the same index. 
% This holds, for instance, when the group is hyperbolic relative to abelian subgroups. 

\end{abstract}

\section{Introduction}\label{section:introduction}
A group $G$ is {\em finite index rigid} if every two isomorphic subgroups of finite index have the same index in $G$. For example, the group $\bbZ$ is \emph{not} finite index rigid, since it has two isomorphic subgroups $2\bbZ\cong3\bbZ$ with $[\bbZ:2\bbZ]\neq[\bbZ:3\bbZ]$. Currently known families of finite index rigid groups include finite groups, groups with non-zero Euler characteristic (such as non-abelian free groups and hyperbolic surface groups), groups with some non-zero $l^2$-Betti number, lattices in simple Lie groups \cite{mostow1973strong} (such as fundamental groups of finite-volume hyperbolic n-manifolds), groups with infinitely many ends \cite{sykiotis2018complexity}, and most recently, non-elementary hyperbolic groups \cite{lazarovich2023finite}.

Our main result is the following. We say that an infinite finitely generated group $G$ is \emph{NRH} if it is not hyperbolic relative to any collection of proper subgroups (for the definition of relatively hyperbolic groups, see \ref{Def: Relatively hyperbolic grups}).

\begin{theorem}
\label{thm:nrh}
  Let a group $G\not\simeq\bbZ$ be torsion-free, hyperbolic relative to type-F and NRH proper subgroups $\calP$. Then $G$ is finite index rigid.
\end{theorem}

There are many classes of groups known to be NRH, including for instance non-virtually-cyclic nilpotent groups \cite{dructu2005tree}, and using known results in this regard we obtain:

\begin{corollary}
\label{cor:list}
The following are finite-index rigid:
\begin{enumerate}
\item fundamental groups of complete finite-volume manifolds of pinched negative curvature,
\item non-abelian limit groups,
    \item torsion-free group hyperbolic relative to nilpotent subgroups,
    \item free-by-cyclic groups with respect to an exponentially growing automorphism.
\end{enumerate}
\end{corollary}

% Our main result is the following.
% \begin{theorem}\label{Thm: nice result}
%     Let $G\not\cong\bbZ^n$ be a torsion-free toral relative hyperbolic group, then $G$ is finite index rigid.
% \end{theorem}

% We say that a torsion-free group $G$ is \emph{toral relatively hyperbolic} if $G$ is hyperbolic relative to a finite collection of finitely generated abelian subgroups (for the definition of relatively hyperbolic groups, see \ref{Def: Relatively hyperbolic grups}) . 

% Examples of relatively hyperbolic groups include: (1) The fundamental groups of complete finite volume hyperbolic manifolds are hyperbolic relative to the cusp subgroups (which are finitely generated free abelian groups). 
% (2) non-abelian limit groups \cite{dahmani2003combination}, and (3) %Some 
% free-by-cyclic groups are hyperbolic relative (not necessarily toral) to subgroups which are free-by-cyclic for an automorphism of polynomial growth. \cite{dahmani2022relative}.

A common strategy for proving that a group is finite index rigid is to associate to each of its subgroups some quantity which is proportional to its index in $G$. 

Let us demonstrate this strategy in the context of fundamental groups of finite-volume hyperbolic $n$-manifolds. Let $M$ be a complete hyperbolic $n$-manifold of finite volume. Then the fundamental group $\Gamma = \pi_1(M)$ can be seen as a lattice in the simple Lie group $G = \Isom(\bbH^n)$ of isometries of the $n$-dimensional hyperbolic space, and $M = \bbH^n / \Gamma$. 
Mostow's Rigidity Theorem \cite{mostow1968quasi,prasad1973strong} states that, for $n\geq 3$, isomorphic lattices in $G$ are conjugate in $G$. 
Therefore, if $\Gamma_1, \Gamma_2\leqslant\Gamma$ are isomorphic finite index subgroups, then in particular $\vol{\bbH^n/\Gamma_1}=\vol {\bbH^n/\Gamma_2}$. Observe  that the volume is multiplicative in the index, i.e. 
$$[\Gamma:\Gamma_1]\cdot\vol{\bbH^n/\Gamma}=\vol{\bbH^n/\Gamma_1}=\vol {\bbH^n/\Gamma_2}=[\Gamma:\Gamma_2]\cdot\vol{\bbH^n/\Gamma},$$
and so $[\Gamma:\Gamma_1] = [\Gamma:\Gamma_2]$.

This example is of particular interest to us since $\Gamma = \pi_1(M)$ is a torsion free group which is relatively hyperbolic to subgroups isomorphic to $\bbZ^{n-1}$. Note that \Cref{cor:list}-(1) is more general. 

Following the strategy demonstrated above, we introduce a complexity invariant, from which we will be able to manufacture a multiplicative invariant.

\begin{definition}\label{notation: complexity}
 Given a group pair ($G$, $\mathcal{P}$), a \emph{(simplicial) relative classifying space} is a pair $(\bar{X}, \bar{\mathcal{A}})$ consisting of a classifying space $\bar{X}$ for $G$ (i.e. $\pi_1(\bar{X}) = G$ and $\pi_n(\bar{X}) = 0$ for all $n \geq 2$) and a collection of disjoint subcomplexes $\bar{\mathcal{A}} = \{\bar{A}_1, \dots, \bar{A}_n\}$ such that each $\bar{A}_i$ is a classifying space for $P_i$, and the inclusion $\bar{A}_i \hookrightarrow \bar{X}$ induces the inclusion $P_i \hookrightarrow G$ in $\pi_1$.

 We define the \emph{complexity} $\C(G, \mathcal{P})$ to be the minimal number of cells in a simplicial relative classifying space $Y$ for $(G, \mathcal{P})$. We refer to the number of cells in $Y$ as the \emph{volume} of $Y$, denoted $\vol{Y}$.
\end{definition}

\begin{theorem}\label{Theorem: Main}
    Let $G$ be a torsion-free, one-ended hyperbolic group relative to type-F proper subgroups $\calP$. Then there exists $\alpha=\alpha(G)$ such that for all $H\leqslant G$ of finite index
    \begin{equation}\label{difficult inequality}
    \alpha[G:H]\leq \C(H,\calP_H)
    \end{equation}
    where $\calP_H$ are the induced peripheral on $H$ (see \Cref{Def:Induced peripheral structure on subgroups}).
\end{theorem}

A well-known result by Rips states that a torsion-free hyperbolic group admits a finite classifying space \cite{gromov1987hyperbolic}. 
Dahmani \cite{dahmani2003classifying} extended this to the relative hyperbolic case: if a torsion-free group $G$ is hyperbolic relative to some finite collection of type-$F$ subgroups, then it admits a finite relative classifying space. 
We therefore conclude that 
$\C(G,\calP)<\infty.$

Let $G$ be a torsion-free hyperbolic group relative to subgroups $\calP$ of type-F. Let $Z$ be a relative classifying space for $(G,\calP)$ such that $\C(G,\calP) = \vol{Z}$, then the cover that corresponds to $H\leqslant G$ is a relative classifying space for $(H,\calP_H)$.
We then get the following inequality $$\C(H,\calP_H)\le \C(G,\calP)[G:H].$$

We have thus managed to find some $\beta=\beta(G)$ such that for all $H\leqslant G$ of finite index 
    \begin{equation}\label{easy inequality}
        \C(H,\calP_H)\le\beta[G:H] 
    \end{equation}

\begin{theorem}\label{thm:fin-ind}
    Let a group $G\not\simeq \bbZ$ be torsion-free, hyperbolic relative to type-F proper subgroups $\calP$.
    Then, if $H,H'$ are finite index subgroups of $G$ such that $(H,\calP_H)\simeq (H',\calP_{H'})$ then $[G:H]=[G:H']$.
\end{theorem}
Here by an isomorphism $(H,\calP)\simeq (H',\calP')$ we mean, roughly speaking, that the group isomorphism $H\cong H'$ preserves the peripheral structure, see \Cref{defn:pair}.

\Cref{thm:fin-ind} is implied by \Cref{Theorem: Main} and Inequality \eqref{easy inequality}, using an argument that we sketch here and spell out in \S\ref{Section: proof of main}.
Following Reznikov \cite{reznikov1995volumes}, one can define a new group pair invariant by considering the limit inferior over the terms $\frac{\C(H,\calP_H)}{[G:H]}$, where $H\subseteq G$ ranges over all finite index subgroups.
This yields a multiplicative isomorphism invariant of group pairs, which is positive (by \Cref{Theorem: Main} and \Cref{easy inequality}).
The strategy demonstrated above then applies to show that, for isomorphic group pairs, the corresponding isomorphic subgroups have the same index.  

Finally, we argue that for groups as in \Cref{thm:nrh}, isomorphisms between finite-index subgroups necessarily preserve the peripheral subgroups.
Then, applying \Cref{thm:fin-ind}, we conclude that such groups are finite index rigid.

Most of the remainder of this paper is dedicated to the proof of \Cref{Theorem: Main}. 
We begin with a compact simplicial complex $K$ that realizes the complexity $\C(H, \calP_H)$, along with the Cayley graph $X$ of $G$.
By attaching peripheral "cusps" to $X$, one can constructs a space $\cusped{X}$ that is $\delta$-hyperbolic.
Groves and Manning \cite{groves2008dehn} constructed such a graph that also admits Mineyev’s bicombing (which we recall in \S\ref{sec:pre}).

In \S\ref{Section: Resolution}, following Lazarovich \cite{lazarovich2023finite}, we use the bicombing to define a singular weighted pattern $\bar \calF$ on the $2$-skeleton of $\bar K$.
In \S\ref{Section: Bounding the total weight}, we prove that, when we reduce to a subpattern depending on some $R_0\in \bbR$, the total weight of the pattern $\bar\calF_{\leq R_0}$ is bounded above by the complexity of $K$.

In \S\ref{Subsec: cylindrical cusped space} we discuss relative group cohomology and cylindrical cusped spaces, mostly following \cite{manning2020cohomology}. In \S\ref{Section: Quasi surjectivity}, we use cohomological methods of \cite{manning2020cohomology}
to show that a continuous $H$-equivariant map $\Phi: \cusped{K}\to \cusped{X}$
is quasi-surjective, with a constant $R$ depending only on $G$. 
In \S\ref{Section : construction of q.s map}, we use the contractibility properties of Rips complexes to prove that one can construct such a map with the additional property that $X$ can be found in a neighborhood of the image of $\Phi$, even when $\Phi$ is restricted to the $1$-skeleton of $K$. 
Finally, in \S\ref{Section: proof of main} we combine all of these results to prove that the total weight of the pattern $\bar \calF_{\le R}$ is bounded below by the volume of $X/H$, and by above by the complexity of K.
In other words, we show that there exists $\alpha$ and $\beta$ now depending only on $G$, such that \begin{equation}\label{equ: bounded weight}
    \frac {1}{\alpha}\cdot \vol{X/H
 }\leq \text{the total weight of } \bar \calF_{\leq R} \leq \beta \cdot \C(H,P_H).
\end{equation} 
\Cref{Theorem: Main} then follows, and we finish by showing how \Cref{thm:fin-ind} ---and subsequently \Cref{thm:nrh}--- are deduced.

\section{Preliminaries}\label{sec:pre}
\subsection{Basic definitions and notations}
\begin{definition}
    Let $X$ be a geodesic metric space. 
    We use $\geod{x,y}$ to denote some geodesic with endpoints $x,y$. 
    We use $[x,y]$ to denote all points on all such geodesics.
    
    Let $\Delta$ be a geodesic triangle in $X$ with sides $\geod{x,y},\geod{y,z}, \geod{z,x}$ . We say that $\Delta$ is  $\delta$-thin if for every point $m\in \geod{x,y}$ there is a point on $\geod{y,z}\cup\geod{z,x}$ at distance at most $\delta$ of $m$ (and similarly for points on other sides of $\Delta$). 
    
    We say that $X$ is $\delta$-hyperbolic  if  there exists some $\delta>0$ such that every geodesic triangle in $X$ is $\delta$-thin \cite{gromov1987hyperbolic}.
\end{definition}

\subsection{Relative hyperbolic groups}
\begin{definition}
\label{defn:pair}
    A \emph{group pair} $(G,\calP)$ is a group $G$ and a collection $\calP = \{P_1,\dots,P_n\}$ of subgroups $P_i\le G$. 
    
    Two group pairs $(G_1,\calP_1),(G_2,\calP_2)$ are isomorphic if there exists an isomorphism $\phi:G_1\to G_2$ and a bijection $\pi:\calP_1 \to\calP_2$ such that for all $P\in \calP_1$, the image $\phi(P)$ is conjugate to $\pi(P)\in \calP_2$.
\end{definition}

Intuitively, a group pair $(G,\calP)$ is relatively hyperbolic if its Cayley graph is hyperbolic outside the left cosets of $P_1,\dots,P_n$. 
These groups where defined by  Bowditch \cite{bowditch2012relatively} and by Farb \cite{farb1998relatively} based on ideas of Gromov \cite{gromov1987hyperbolic}.
We will use an equivalent definition given by Groves and Manning  \cite{groves2008dehn} which we recall in \Cref{Def: Relatively hyperbolic grups}.

\begin{definition}\label{Def: combinatorial horoball}
    Let $\Gamma$ be a graph.
    The \emph{combinatorial horoball based at $\Gamma$} denoted $\calH:=\calH(\Gamma)$, is a graph with vertex set  $\calH^{(0)}:=\Gamma^{(0)}\times\bbN_0$ and the following edges
    \begin{enumerate}
        \item A \emph{vertical} edge between $(v,m)$ and $(v,m+1)$ for all vertices $v$ of $\Gamma$.
        \item A \emph{horizontal} edge between $(v,m)$ and $(w,m)$ whenever $m$ is an integer and the vertices $v,w$ of $\Gamma$ satisfy $d_\Gamma(v,w)\leq 2^m$.  
    \end{enumerate}
\end{definition}

\begin{definition}
    We define the depth of a vertex $v=(\bar v,m)\in\calH$ to be $\depth(v)=m$.
\end{definition}

\begin{definition}
    Let $G$ be a group, $\calP = \{P_1,\dots,P_n\}$.
    Let $X$ be a connected graph, and $\calA = \{A_1,\dots,A_n\}$ be a collection of connected subgraphs. We write $(G,\calP) \actson (X,\calA)$ if $G$ acts on $X$ and for each $1\le i \le n$, $P_i$ preserves $A_i$. 
    The action is \emph{geometric} if both $G\actson X$ and $P_i\actson A_i$ are geometric actions.
\end{definition}

Note that when $(G,\calP)$ is a finitely generated pair, one can choose a generating set $S$ for $G$ such that $S\cap P_i$ generates $P_i$.
By taking $X$ to be the Cayley graph of $(G,S)$ and $A_i$ to be the Cayley graph of $(P_i,S\cap P_i)$ (as a subgraph of $X$), we get that the action $(G,\calP) \actson (X,\calA)$ is geometric.
In this case, the space $\comb{X,\calA}$ defined next is the \emph{combinatorial cusped graph} defined in \cite{groves2008dehn}. 

\begin{definition}\label{comb{X}}
    When $(G,\calP)$ acts geometrically on $(X,\calA)$, define $$\comb{X,\calA} = X \cup \bigcup_{1\le i\le n,\; g\in G/P_i}\calH(gA_i)$$
    gluing $gA_i\subset X$ with $gA_i\times \{0\}\subset\calH(gA_i)$. 
\end{definition}

\begin{definition}\label{Def: Relatively hyperbolic grups}
    Let $(G,\calP)$ be a finitely generated pair.
    We say that $(G,\calP)$ is relatively hyperbolic if $(G,\calP)$ acts geometrically on some graph pair $(X,\calA)$ such that $\comb{X,\calA}$ is $\delta$-hyperbolic for some $\delta>0$. 
\end{definition}

It is standard to show that being relatively hyperbolic is independent of the graph pair $(X,\calA)$, and so it coincides with the definition of Groves and Manning  \cite[Definition 3.12]{groves2008dehn}.

\medskip

\textbf{Throughout the text we fix the group pair $(G,\calP)$ and a graph pair $(X,\calA)$ on which $(G,\calP)$ acts geometrically.}
We will often write $\comb{X}$ to refer to $\comb{X,\calA}$.
\medskip

Given a relatively hyperbolic group $(G,\calP)$, any finite index subgroup $H\leqslant G$ is relatively hyperbolic with respect to the induced collection
\begin{equation*}\label{Def:Induced peripheral structure on subgroups}
\calP_H=\{H\cap aPa^{-1}|P\in\calP, a\in B_{HP} \}
\end{equation*}
where $B_{HP}$ is a set of representatives of double cosets $\{HgP, g\in G\}$.
\begin{definition}
Let $(G,\calP)$ be a relatively hyperbolic pair. The Gromov boundary of $\comb{X,\calA}$ is denoted $\partial(G,\calP)$, and called the \emph{Bowditch boundary}.
\end{definition}
Note that the induced structure on finite index subgroups yields the homeomorphism $$\partial(H,\calP_H)\cong \partial(G,\calP).$$

\subsection{The Mineyev bicombing}\label{Pre: Mineyev's bicombing}
\begin{notation}
    For a simplicial complex $X$, we denote by $C_n(X)$ its rational $n$-chains. For $a=\sum_\sigma a_\sigma\sigma\in C_n(X)$ we denote $\norm{a}_1=\sum_\sigma\abs{a_\sigma}$ and $\norm{a}_{\infty}=\max_\sigma|a_\sigma|$. We assume that $a_\sigma \ge 0$.

    For a subset $A\subseteq X$, we denote by $\calN_\delta(A)$ the set of edges of $X$ which are at Hausdorff distance at most $\delta$ from some element in $A$. 

    We denote by $X^n$ the collection of $n$-simplices of $X$, and by $X^{(n)}$ its $n$-skeleton.
\end{notation}
\begin{definition}
\label{Def: Globally stable bicombing GSB}
Let $X$ be a simplicial graph, $G\actson X$ a simplicial action. A map $q:X^0\times X^0\to C_1(X)$ is called a \emph{globally-stable bicombing} on $X$ if it satisfies the following properties: 
\begin{enumerate}[label=(B\arabic*), ref=(B\arabic*)]
    \item $q$ is a \emph{$\bbQ$-bicombing}, i.e. for all $x,y\in X^0$, $\partial q(x,y)=y-x$ and $q(x,x)=0.$
    \item\label{Def: quasi-geodesic} $q$ is \emph{quasigeodesic}, i.e. for all $x,y\in X^0$, and geodesic $\geod{x,y}$, there exists $\delta=\delta(X)$ such that $\supp q(x,y)\subseteq\calN_{\delta}(\geod{x,y})$ and $\norm{q(x,y)}_1\leq\delta \cdot d(x,y)$.
    \item $q$ is \emph{$G$-equivariant}, i.e. for all $x,y\in X^0$, $g.q(x,y)=q(g.x,g.y)$.
    \item $q$ is \emph{antisymmetric}, i.e. for all $x,y\in X^0$, $q(x,y)=-q(y,x)$.
    \item\label{Def: bounded defect}  $q$ has \emph{bounded defect}, i.e. for every $x,y,z\in X^0$, there exists $\delta=\delta(X)$ such that  $\norm{q(x,y)+q(y,z)+q(z,x)}_1\leq\delta$.
\end{enumerate}
\end{definition}

% \begin{theorem} [Mineyev{\cite[Theorem 10]{mineyev2001straightening}}]
    % Every locally finite hyperbolic graph with cocompact group action, supports a globally stable bicombing.
% \end{theorem}
% The following is a generalization of Mineyev {\cite[Theorem 10]{mineyev2001straightening}}

Globally-stable bicombing exist for hyperbolic groups (Mineyev {\cite[Theorem 10]{mineyev2001straightening}}) and relatively-hyperbolic groups (Groves-Manning{\cite[Part 1, section 6]{groves2008dehn}}).

\begin{theorem}[Mineyev, Groves-Manning]\label{rem: bicomb}
     Let $(G,\calP)$ be a relatively hyperbolic pair, and let $\comb{X}$  be a combinatorial cusped graph for $(G,\calP)$. Then there exists a globally-stable bicombing on $\comb{X}$.
\end{theorem}

Throughout the text, we fix a globally stable bicombing $q$ on the fixed locally finite hyperbolic graph $\comb{X}$, corresponding to the pair $(G,\calP)$. Any such bicombing has the following properties.

\begin{lemma}[{\cite[Lemma 3.5]{lazarovich2023finite}}] \label{lem : q-bound} 
    There exists $\delta=\delta(X)$ such that for all $x,y\in X^0$,  $\norm{q(x,y)}_\infty\leq \delta$.
\end{lemma}

\begin{lemma}[{\cite[Lemma 3.6]{lazarovich2023finite}}]\label{lem : q-unbound} There exists $\rho$ such that for all $x\neq y\in X^0$ and for all $z\in[x,y]$, the restriction $a_z=q(x,y)|_{\calN_\rho(z)}$ of $q(x,y)$ to the edges in the ball of radius $\rho$ around $z$ satisfies $\norm{a_z}_1\geq1$. 
\end{lemma}

In \Cref{Section: Resolution,Section: Bounding the total weight},
we will restrict our claims to a cusped space with vertices at depth at most $R$, denoted $\comb{X}_{\le R}$. Note that $G$ acts cocompactly on $\comb{X}_{\leq R}$. In this case, we have the following result

\begin{corollary}[{\cite[Corollary 3.7]{lazarovich2023finite}}]\label{cor : <1/lambda} There exists $\rho=\rho(X)$ such that for all $R\ge 0$ there exists $\lambda=\lambda(X,R)$,  such that for all $x\neq y\in X^0$ and for all $z\in[x,y]\cap \comb{X}_{\le R}$, there exists an edge in $B_\rho(z)$ whose coefficient in $q(x,y)$ is at least $1/\lambda$.
\end{corollary}

\begin{proof}
By \Cref{lem : q-unbound}, there exists some $\rho=\rho(X)$ such that  $\norm{q(x,y)|_{\calN_\rho(z)}}_1\geq 1$. Since $\comb{X}$ is locally finite, balls have finitely many edges. Now, $G$ acts cocompactly on $\comb{X}_{\leq R}$, so there is a maximal number of edges in a ball of radius $\rho$.
\end{proof}

\subsection{Weighted patterns}
Lazarovich \cite{lazarovich2023finite} introduced the notion of a \emph{weighted singular pattern}. Starting from an action of a group $G$ on a hyperbolic space supporting a globally stable bicombing, he demonstrated how one can define a weighted pattern on a $2$-dimensional simplicial complex $K$, with $\pi_1(\bar{K})=G$.

\begin{definition}[{\cite[Definition 4.1]{lazarovich2023finite}}]
    Let $K$ be a $2$-dimensional simplicial complex. A \emph{singular pattern} $\calF$ on $K$ is an immersed graph $\calF\immers K$ such that:
    \begin{enumerate}
        \item The vertices of $\calF$ are called \emph{connectors} and they are sent injectively to the interiors of edges and $2$-simplices of $K$. They are called \emph{regular} or \emph{singular} accordingly. See \Cref{fig: singular pattern}.
        \item The edges of $\calF$ are called \emph{segments}, and their interiors are sent to straight line segments in the interior of $2$-simplices of $K$. Segments that connect regular connectors are called \emph{regular}, while other segments are called \emph{singular}. Every singular segment has a regular connector as one of its endpoints, and it does not meet any other segment in the interior of the $2$-simplex in which it is contained.
        \item There are finitely many connectors and segments in each simplex.
        \item Each connector $c$ is the endpoint of exactly one segment in each $2$-simplex containing $c$ in its closure.
    \end{enumerate}
    Connected components of $\calF$ are called \emph{tracks}. A \emph{regular pattern} is a pattern without singular connectors. 
\end{definition}

\begin{figure}
    \centering
    \includegraphics{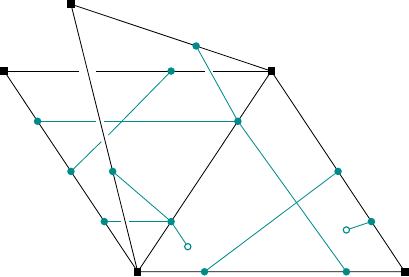}
    \caption{A singular pattern (in cyan) on a 2-complex (in black) consisting of three 2-simplices with a common edge. The regular connectors are shown by full circles, and singular connectors by empty circles.}
    \label{fig: singular pattern}
\end{figure}

\begin{definition}[{\cite[Definition 4.2]{lazarovich2023finite}}] A \emph{weighted singular pattern} on a finite simplicial complex $K$ is a pair $(\calF,\bfw)$ where $\calF$ is a singular pattern, and $\bfw$ is a function from the set of connectors to $\bbR_{\ge0}$ such that the weight of a singular connector $\bfc$ is $\bfw(\bfc)=0$.
Let us also define:
\begin{enumerate}
    \item The weight of a track $\bft$ is the maximal weight of its connectors $\bfw(\bft):=\max_{\bfc\in \bft} \bfw(\bfc)$.
    \item The total weight of $\calF$ is the sum of weights of its tracks $\bfw(\calF)=\sum_{t\subseteq\calF}\bfw(\bft)$.
    \item The \emph{defect} of a segment in a weighted singular pattern is $\df_\bfw(\bfs)=\abs{\bfw(\bfc_1)-\bfw(\bfc_2)}$ where $\bfc_1$, $\bfc_2$ are the endpoints of $\bfs$.
    \item The defect of a weighted pattern $(\bfw,\calF)$ is the sum of defects of its segments, $\df_\bfw(\calF)=\sum_{\bfs\subseteq\calF}\df_\bfw(\bfs)$. 
    % (equivalently, the sum of defect of all segments).
    \item A weighted singular pattern is \emph{perfect} if it is regular and $\df(\calF)=0$. 
\end{enumerate}
\end{definition}

\begin{lemma}[{\cite[Lemma 4.3]{lazarovich2023finite}}] \label{Lem: perfect resolution}
Let $(\calF,\bfw)$ be a weighted pattern on a compact simplicial complex $K$. Then there is a perfect weighted pattern $(\calF',\bfw')$ such that $\calF'\subseteq\calF$, $\bfw'\leq\bfw$ and $\bfw(\calF)\leq\bfw'(\calF')+\df_\bfw(\calF)$. 
\end{lemma}

\section{Resolution of globally stable bicombings} \label{Section: Resolution}

Let $(G,\calP)$ be a relatively hyperbolic group pair and let $K$ be a simply-connected 2-dimensional simplicial complex with a free and cocompact simplicial $G$-action.

Let $X$ be a graph with $\cusped{X}=\comb{X}$ a $\delta$-hyperbolic graph quasi isometric to the combinatorial cusped space (on which $G$ acts isometrically and properly).
Following \Cref{rem: bicomb}, we fix some globally stable bicombing $q$ on $\comb{X}$. 
Let $\Phi:K^0\to X^0$ be a $G$-equivariant map.
We define a map $K^1\to C_1(\cusped{X})$ via the bicombing $q$, i.e for $e\in K^1$, $q(e)=q(\Phi(e_-),\Phi(e_+))$.

 \begin{definition}[{\cite[Definition 5.1]{lazarovich2023finite}}]\label{Def: Resolution}
     A \emph{resolution} of $q$ to $K$ via $\Phi$ is a $G$-invariant weighted singular pattern $(\calF,\bfw)$ on K together with a bijection for every edge $e\in K^1$ $$\bfc(e,\cdot):\supp q(e) \to \set{\text{regular connectors of } \calF \text{ on } e}$$
     such that the following holds for every edge $e\in K^1$:
     \begin{enumerate}[label=(R\arabic*), ref=(R\arabic*)]
    \item\label{R-symmetry}[symmetry] $\bfc(-e,-f)=\bfc(e,f)$ for all $f\in \supp q(e)$.
    \item\label{R-G-equivaraint}[$G$-equivariance] $g\cdot \bfc(e,f)=\bfc(g.e,g.f)$ for all $g\in G$ and $f\in\supp q(e)$. 
    \item\label{R-quasi-ordered}[quasi-ordered] for all $f,f'\in\supp q(e)$, if there exists $\delta=\delta(X)$ s.t $$d(\Phi(e_-),f)<d(\Phi(e_-),f')-\delta$$ then when considering $e$ as an interval ordered from $e_-$ to $e_+$ we have $\bfc(e,f)<\bfc(e,f')$ on $e$.  
    \item\label{R-weight} For all $f\in\supp q(e)$ the weight $\bfw(c(e,f))$ is the coefficient of $f$ in $q(e)$.
    \item\label{R-2-simplex} For every $2$-simplex $\Delta\in K^2$, two regular connectors $\bfc$, $\bfc'$ on the boundary of $\Delta$ are connected by a regular segment if and only if we can write $\bfc=\bfc(e,f),\, \bfc'=\bfc(e',-f)$ such that $e,\, e',\, e''$ are the edges of $\partial\Delta$, oriented such that $\partial\Delta=e+e'+e''$, the edge $f$  belongs to $\supp q(e)$, $-f$ belongs to $\supp q(e')$, and neither of $\pm f$ belongs to $\supp q(e'')$. 
 \end{enumerate}    
 \end{definition}

\begin{lemma}
     For every $G, \, K, \,\Phi,\, X,\, q$ as above, there exists a resolution of $K$ via $\Phi$.
\end{lemma}

\begin{proof}
We are in a situation similar to \cite[Lemma 5.3]{lazarovich2023finite}, with the exception that we do not necessarily have a cocompact action $G\actson \cusped{X}$.
Yet, it is sufficient for this construction that $G\actson K$ is cocompact.
Thus, we can start by choosing orbit representatives of vertices, edges and 2-simplices in $K$, defining the map $\bfc(e,\cdot)$ and the connectors and segments just as in \cite[Lemma 5.3]{lazarovich2023finite}, relying only on the properties of the bicombing $q$.
Extending $\bfc(e,\cdot)$ G-equivariantly to all $\sigma\in K^2$ we obtain a resolution of $q$ to $K$ via $\Phi$. 
\end{proof}
\begin{remark}
    Note that in our case, given $e\in K^1$, we may have connectors $\bfc(e,f)$ on $e$ for which $f$ is in $\cusped{X}\setminus X$.
\end{remark}

\begin{lemma}[{following \cite[lemma 5.5]{lazarovich2023finite}}]\label{lem: T1-T4}
Let $(\calF,\bfw)$ be a resolution of $q$ to $K$ via $\Phi$, then the following holds
\begin{enumerate}[label=(T\arabic*), ref=(T\arabic*)]
\item\label{T1} Each track of $\calF$ is embedded and meets each edge of $K$ at most once
\item\label{T2} Each track corresponds to a unique unoriented edge $f$ in $\cusped{X}$.
\item\label{T3} Each track is compact and its stabilizer is finite.
\item\label{T4} If two tracks $\bft$, $\bft'$ in $\calF$ intersect then there exists some $\delta=\delta(X)$ such that the corresponding edges $f,\, f'$ in $\cusped{X}$ satisfy $$d(f,f')\leq \delta.$$
\end{enumerate}
\end{lemma}

We follow the outline of the proof provided in \cite{lazarovich2023finite}, briefly explaining why it readily carries over to our slightly different setting.

\begin{proof}
Properties \ref{T1}-\ref{T2} holds just by \Cref{Def: Globally stable bicombing GSB}.
Property \ref{T3} holds: as the map that sends a track to its corresponding unoriented edge in $\cusped{X}$ is $G$-equivariant, and the action $G\actson \cusped{X}$ is free.
Thus, the stabilizer of each track is trivial.
To show that each track is compact, notice first that since the action $G\actson K$ is cocompact, the action of $G$ on the set of connectors in $\calF$ has finitely many orbits.
Furthermore, the map that assigns to each connector in $\calF $ a corresponding unoriented edge in $\cusped{X}$ is $G$-equivariant.
It follows that each track has finitely many connectors, and so is compact. 

The proof of property \ref{T4} is done as in \cite{lazarovich2023finite}, and holds just the same as it requires only the properties of the bicombing and the definition of the resolution.
\end{proof}

From here on, we will work with resolutions for which the regular and singular connectors originate from edges in $\supp q(e)\subseteq \cusped{X}$ of depth at most $R$, where $R$ is some arbitrary constant which will be chosen later to depend only on $X$. 

\begin{notation}
    Let $R\in\bbR$. We denote by $\calF_{\le R}$ the restriction of $\calF$ to the pattern that is obtained by including only connectors induced by edges in $\cusped{X}_{\le R}$.
    As $\calF$ is $G$-invariant, it gives rise to a resolution  $\bar\calF$ on $\bar K=K/G$.
\end{notation}

\section{Bounding the total weight}\label{Section: Bounding the total weight}
As explained in the introduction, the proof of \Cref{Theorem: Main} goes through proving the inequality \eqref{equ: bounded weight}.
In this section, we focus on the right-hand inequality. 

We will adopt the notation $\boldsymbol{\delta}$ introduced in \cite{lazarovich2023finite}: 
\begin{notation}
   We denote by $\boldsymbol{\delta}$ some positive constant that depends only on $X$.
   Given $R\ge 0$, we denote by $\boldsymbol\delta_R$ some positive constant that depends only on $X$ and $R$.
\end{notation}

\begin{proposition}\label{prop;upper bound}
    Let $G$ be a one-ended group, hyperbolic relative to $\calP$, and fix some globally stable bicombing on $\comb{X}$.
    For every simply connected $2$-dimensional simplicial complex $K$ with a free cocompact simplicial action $G\actson K$, there exists a resolution $\bar\calF$ of $q$ to $\bar{K}:=K/G$ satisfying for every $R\in\bbR$ 
    \begin{equation} \label{equ : w<v}
        \bfw (\bar\calF_{\le R})\leq \boldsymbol\delta_R\cdot \vol{\bar{K}}.
    \end{equation}
\end{proposition}

\begin{proof}
    Following \cite[Proposition 6.1]{lazarovich2023finite}, we let $v_1,\dots,v_l,\,e_1,\dots e_m$ and $\Delta_1,\dots , \Delta_n$ be representatives for vertices, edges and 2-simplices respectively, for the action $G\actson K$. Assume that a $G$-equivariant map $\Phi:K^0 \to X^0$ is with a minimal displacement among all such maps.
    That is: 
    \begin{equation}
        \sum_{i=1}^m d(\Phi((e_i)_-),\Phi((e_i)_+))
    \end{equation}
    is minimal.

Denote by $(\calF,\bfw )$ and $(\bar{\calF},\bfw)$
the resolutions of $q$ to $K$
and to $\bar{K}$ via $\Phi$.
By \Cref{Lem: perfect resolution}, there exists a perfect weighted resolution $(\bar{\calF}',\bfw')$ such that $\bar{\calF}'\subseteq\bar{\calF}$, $\bfw'\leq\bfw$ and $$\bfw(\bar\calF_{\le R})\leq\bfw'(\bar\calF_{\le R}')+\df_\bfw(\bar\calF_{\le R})$$

We show that $\bar{\calF}$ satisfies \Cref{equ : w<v} by proving the following claims.
 \begin{claim}
     $\df(\bar\calF_{\le R})\leq\boldsymbol\delta_R\cdot n$.
 \end{claim}
     
     \begin{proof}
         Let $\Delta\in K^2$ with vertices $v,v',v''$ and edges $e,e',e''$ (and $\partial\Delta=e+e'+e''$). Considering the collection of edges in $\cusped{X}_{\leq R}\cap\supp q(e)\cap\supp q(e')\cap\supp q(e'')$, we partition them into two classes.  set
    \begin{equation*} 
         \begin{aligned}
         \calA_\Delta ={} & \{f\in \cusped{X}_{\leq R}^1 : (f\in \supp q(e) \lor -f\in \supp q(e))\\
         & \land(f\in \supp q(e') \lor -f\in \supp q(e')) \\
         & \land(f\in \supp q(e'') \lor -f\in \supp q(e''))\} 
        \end{aligned}
    \end{equation*}
         
    Since $q$ is quasigeodesic \ref{Def: quasi-geodesic}, each $f\in \calA_\Delta$ is at distance $\boldsymbol{\delta}$ from the sides of a geodesic triangle.
    Therefore, the elements of $\calA_\Delta$ are contained in a $\boldsymbol{\delta}$-ball around the center of a geodesic triangle.
    By  the local compactness of $\cusped{x}$, and since we are considering only edges of depth at most $R$, it follows that $\abs{\calA_\Delta}<\boldsymbol\delta_R$.
    By construction, each $f\in \calA_\Delta$ is  associated with three singular segments  in $\Delta$, each of which has a universally bounded defect (\Cref{lem : q-bound}).
    We conclude that the sum of the defects of all segments corresponding to edges in $\calA_\Delta$ is at most $\boldsymbol\delta_R$.
  
    Now considering edges  $f \notin \calA_\Delta$, we have three cases to consider: either $f$ is in exactly one of $\supp q(e)$, $\supp q(e')$ and $\supp q(e'')$, or $f$ is in one of these supports, and $-f$ in another, or $f$ is in exactly two of these supports. 
         In any case, the defect of the segments corresponding to $f$ is equal to the absolute value of the coefficient of $f$ in the sum $q(e)+q(e')+q(e'')$. Hence the sum of defects of all edges in $\Delta$ is bounded by $\norm{q(e)+q(e')+q(e'')}_1$. It remains to recall that $q$ has bounded defect \ref{Def: bounded defect}, i.e. $\norm{q(e)+q(e')+q(e'')}_1\leq\boldsymbol{\delta}$ .
         
         It follows that the sum of the defects of the segments of $\bar\calF_{\leq R}$ in $\bar{\Delta}$ is bounded by $\boldsymbol\delta_R$, and hence $\df(\bar\calF_{\leq R})\leq n\cdot\boldsymbol\delta_R$ 
      \end{proof}

Denote by $\calF'$ the pattern on $K$ obtained by considering all lifts of the pattern $\bar\calF'$.
Every track $\bft$ of $\calF'$ is compact, embedded and locally two-sided. Since $K$ is simply connected, a track $\bft$ separates $K$ into two components.
      
     \begin{claim}
         $\bfw '(\bar\calF'_{\le R})\leq\boldsymbol\delta_R\cdot m$.
     \end{claim}
     
     \begin{proof}
          Since $G$ is one-ended (and $G\actson K$ freely cocompactly), a track $\bft$ separates $K$ to a bounded component $\bft_+$ and an unbounded component $\bft_-$. 
          If a vertex $v_i$ is contained in $\bft_+$ for some track $\bft$ of $\calF’_{\leq R}$, let $\bft(v_i)$ be an outermost track containing $v_i$, otherwise set $\bft(v_i)=v_i$.
          Extend to $\bft(\cdot):K^0\to K^0\cup\{\bft\subset \calF’_{\leq R}\}$ $G$-equivariantly. 
          Set $y_i=\Phi(v_i)$ if $\bft(v_i)=v_i$, set $y_i$ to be an endpoint $(f_i)_+\in\cusped{X}$ of an edge corresponding to the track $\bft$, and $\tild y_i$ the vertical projection of $y_i=(\tild y_i,n)$ onto $X$. 
          Let $\Psi: K^0\to\cusped{X}^0$ and $\tild\Psi:K^0\to X^0$ be the $G$-equivariant maps defined by $\Psi(g.v_i)=g.y_i$, $\tild{\Psi}(g.v_i)=g.\tild y_i$.

          Let $\calF_i'$ be the collection of tracks meeting the edge $e_i$ in $\calF’$, and similarly,  ${\calF’_i}_{\leq R}$ the collection of tracks meeting the edge $e_i$ in $\calF’_{\leq R}$.
          We wish to show that by replacing $\Phi$ with $\tild{\Psi}$ we are able to bound the weight of ${\calF’_i}_{\leq R}$.
          More precisely we will prove that for every $1\leq i \leq m$  
          \begin{equation}\label{equ:weigh bound}
               d(\Psi((e_i)_-),\Psi((e_i)_+))\leq d(\Phi((e_i)_-),\Phi((e_i)_+))-\bfw'({\calF’_i}_{\leq R})/\boldsymbol\delta_R+\boldsymbol{\delta}.
          \end{equation}
          This will follow from the proof of \cite[Claim 6.3]{lazarovich2023finite} (a proof is provided below).
          
         Next, notice that $$d(\tild{\Psi}((e_i)_-),\tild{\Psi}((e_i)_+))-2R\leq d(\Psi((e_i)_-),\Psi((e_i)_+)).$$
          and since we chose $\Phi$ such that is has minimal displacement among all such maps, we get the following inequality 
         \begin{equation*}
         \begin{aligned}    
         \sum_{i=1}^m d(\Phi((e_i)_-,\Phi((e_i)_+))\leq {} & \sum_{i=1}^m d(\tild{\Psi}((e_i)_-),\tild{\Phi}((e_i)_+))\\
         & \leq\sum_{i=1}^m d(\Psi(e_i)_-),\Psi((e_i)_+))+2R \\
         & \leq\sum_{i=1}^m (d(\Phi((e_i)_-),\Phi((e_i)_+))-\bfw'({\calF’_i}_{\leq R})/\boldsymbol\delta_R+\boldsymbol{\delta})+2R)\\
         & \leq\sum_{i=1}^m d(\Phi(e_i)_-),\Phi((e_i)_+))-\sum_{i=1}^m \bfw'({\calF’_i}_{\leq R})/\boldsymbol\delta_R+(\boldsymbol{\delta} + 2R)m  
         \end{aligned}
          \end{equation*}
subtracting $\sum_{i=1}^m d(\Phi(e_i)_-),\Phi((e_i)_+))-\sum_{i=1}^m \bfw'({\calF’_i}_{\leq R})/\boldsymbol\delta_R$ from both sides yields, $$\sum_{i=1}^m \bfw'({\calF’_i}_{\leq R})/\boldsymbol\delta_R\leq (\boldsymbol{\delta} + 2R)m=\boldsymbol\delta_R\cdot m.$$

It remains to prove the that inequality \ref{equ:weigh bound} holds: Let $1\le i\le m$, and denote $e=e_i$, its endpoints $v=e_-$ and $v'=e_+$, and $\Phi(v)=x,\, \Phi(v')=x',\,\Psi(v)=y,\, \Psi(v')=y'$ . 
We wish to prove that $$d(y,y')\leq d(x,x')-\bfw'({\calF’_i}_{\leq R})/\boldsymbol\delta_R+\boldsymbol\delta.$$
Consider $\bft=\bft(v)$ and $\bft'=\bft(v')$. We prove by dividing to cases:

          \begin{enumerate}[label=Case \arabic*]
              \item $\bft=v$, $\bft'=v'$

              By the definition of $\Psi$, both $v$ and $v'$ are not contained in any component $\bfs_+$ for a track $\bfs$ of $\calF_{\leq R}'$.
              In particular, no track $\bfs$ of $\calF'_{\leq R}$ meets $e$, as otherwise either $v\in\bfs_+$ or $v'\in\bfs_+$.
              Thus, $\bfw({\calF'_i}_{\leq R})=0$, and the inequality follows. 
              
              \item $\bft\neq v$, $\bft'\neq v'$.
              \begin{enumerate}[label=Case 2.\arabic*]
              \item $\bft\cap\bft'\neq\emptyset$.
                let $f,f' \in\cusped{X}_{\leq R}$ be the edges corresponding to $\bft$, $\bft'$ respectively, then By property \ref{T4}, $d(y,y')\leq\boldsymbol{\delta}$. 
                Since $(\calF',\bfw')$ is perfect weighted pattern, $\bfw'(t)=\bfw'(c)$ for any connector $\bfc$ on the track $\bft$ in $\calF'$, and by definition of $(\bfw',\calF')$ we also know that for every connector $\bfc$ of $\calF'$ ,$\bfw'(c)\leq\bfw(c)$.
                Summing over all connectors $\bfc$ on the edge $e_i$ we get $$\bfw'({\calF_i'}_{\leq R})=\sum_{\bfc\in e_i}\bfw'(\bfc)\leq\sum_{\bfc\in e_i}\bfw(\bfc)=\norm{q(e)}_1\leq \boldsymbol{\delta}\cdot d(x,x').$$
                Hence, 
                $$0\leq 
                d(x,x')-\bfw'({\calF_i'}_{\leq R})/\boldsymbol{\delta}.$$
                Combining the two inequalities we get 
                $$d(y,y')\leq \boldsymbol{\delta}\leq d(x,x')-\bfw'(\calF_i')/\boldsymbol{\delta}+\boldsymbol{\delta}$$
                
              \item $\bft\cap\bft'=\emptyset$.
              In this case both $\bft$ and $\bft'$ intersect $e$.
              As otherwise, both $\bft_+$ and $\bft'_+$ contain a similar vertex, either $v$ or $v'$, which implies that one of $\bft_+$, $\bft'_+$ contains the other.
              That's in contradiction to the assumption that the map $\bft(\cdot)$ assigns an  outermost track. 
              Let $\bfc$, $\bfc'$ be the connectors in which $\bft$, $\bft'$ meet $e_i$ respectively.
              Then we must have $\bfc < \bfc'$ on $e$. 
              Let $I_1=[v,\bfc]$, $I_2=(\bfc,\bfc']$, $I_3=(\bfc',v']$ be the partition of $e$ into subintervals.
              Let $\calF'_{i,j}$, $1\leq j\leq 3$ be the corresponding partitions of ${\calF_i'}_{\leq R}$ of tracks which meet $I_j$.
              The proof proceeds by bounding the weight of each
              partition $\calF'_{i,j}$.

              If a track $\bft''$ has a connector on $I_2$, then $\bft''_+$ contains one of the endpoints of $e_i$, But as $\bft$ and $\bft'$ are outermost, it follows that each such track $\bft''$ intersects either $\bft$ or $\bft'$. 
              By property \ref{T4}, the edges $f''$ corresponding to the track $\bft''$ is at distance at most $\boldsymbol{\delta}$ from either $f$ or $f'$ (the edges corresponding to the tracks $\bft$, $\bft'$ respectively).
              Since $\cusped{X}$ is locally finite, and since the resolution is composed only from connectors corresponding to edges in $\cusped{X}_{\leq R}$, the number of such edges is at most $\boldsymbol\delta_R$.
              By \Cref{lem : q-bound}, the weigh of $\bft''$ is at most $\boldsymbol{\delta}$, hence we conclude that 
               \begin{equation}
               \bfw'(\calF'_{i,2})\leq\boldsymbol\delta_R.
              \end{equation}
              
              To bound the weight of the remaining two partitions, let $\gamma=\geod{x,x'}$ be the geodesic between $x,x'$ in $\cusped{X}$.
              Since $q$ is quasigeodesic \ref{Def: quasi-geodesic}, there are points $w,w'$ on $\gamma$ at distance at most $\boldsymbol{\delta}$ from $f$, $f'$ respectively. By the definition of a resolution, the edge $f''$ corresponding to a track $\bft''$ in $\calF'_{i,1}$ satisfies $d(x,f'')\leq d(x,f)+\boldsymbol{{\delta}}$. By the triangle inequality, $f''$ is contained in the $\boldsymbol\delta$ neighborhood of the subsegment $\geod{x,w}\subseteq\gamma$.
              As explained in the previous case, we can conclude that there are at most $\boldsymbol\delta_R\cdot d(x,w)$ such edges. And so, by \Cref{lem : q-bound} the weight of $\bft''$ is at most $\boldsymbol\delta_R$.
              And similarly for $\calF_{i,3}$, thus 
              \begin{equation}
              \bfw'(\calF'_{i,1})\leq\boldsymbol\delta_R\cdot d(x,w),\, 
              \, \bfw'(\calF'_{i,3})\leq\boldsymbol\delta_R\cdot d(x',w').
              \end{equation}

              Since $$\bfw'({\calF'_i}_{\le R})=\bfw'(\calF'_{i,1})+\bfw'(\calF'_{i,2})+\bfw'(\calF'_{i,3})$$
              we get 
             \begin{equation}
                 \begin{aligned}
                 d(y,y'){}& \leq\boldsymbol{\delta}+d(w,w')\\
                 &=\boldsymbol\delta+d(x,x')-d(x,w)-d(x',w')\\
                 &\leq \boldsymbol{\delta}+d(x,x')-(\bfw'(\calF'_{i,1})+\bfw'(\calF'_{i,3}))/\boldsymbol\delta_R\\
                 &\leq\boldsymbol{\delta}+d(x,x')-(\bfw'(\calF'_{i,1})+\bfw'(\calF'_{i,2})+\bfw'(\calF'_{i,3}))/\boldsymbol\delta_R\\
                 &\leq \boldsymbol{\delta}+d(x,x')-(\bfw'(\calF'_i))/\boldsymbol\delta_R
                 \end{aligned}
             \end{equation}
              \end{enumerate}
              \item $\bft\neq v, \bft'=v'$.
              This case is done similarly to Case 2.2. The track $\bft$ meets the edge $e$, as otherwise $v'\in\bft_+$.
              Let $\bfc$ be the connector of $\bft$ on $e$.
              We partition $e$ into $[v,\bfc]$ and $(\bfc,v']$ and repeat the argument in Case 2.2. 
          \end{enumerate}

       Thus, we got the desired result, i.e. that for all $1\leq i\leq m$, and $e_i\in K$ an edge representative \begin{equation}
         d(\Psi((e_i)_-),\Psi((e_i)_+))\leq d(\Phi((e_i)_-),\Phi((e_i)_+))-\bfw'(\calF_i')/\boldsymbol\delta_R+\boldsymbol{\delta}.
     \end{equation}

 \end{proof}
Combining the above results we get $$\bfw(\bar\calF_{\leq R})\leq\boldsymbol\delta_R\cdot n+\boldsymbol\delta_R\cdot m+ \leq \boldsymbol\delta_R\cdot \vol{\bar{K}}.$$
This completes the proof of the proposition.
 \end{proof}

\section{Relative group cohomology and the cylindrical cusped space}
\subsection{Cylindrical cusped space}\label{Subsec: cylindrical cusped space}

Bestvina and Mess \cite{bestvina1991boundary} showed how the cohomology of hyperbolic groups relates to the the cohomology of their Gromov boundary.
In \cite{manning2020cohomology}, Manning and Wang extend this result to the relative hyperbolic setting.
Doing so, they define an $N-connected$ cusped space for a type $F_\infty$ pair $(G,\calP)$. As we focus on type $F$ groups, their construction will yield for us a contractible space.
In this section we will define this space, which we call the cylindrical cusped space, and establish its relation to its Rips complex and to the combinatorial cusped space. 

\begin{definition}
    A group is of \emph{type} $F$ if it admits a finite Eilenberg-MacLane classifying space.
    We say that a pair $(G,\calP)$ is of type-F if $G$ and all $P\in \calP$ are of type $F$.
\end{definition}  
Dahmani \cite{dahmani2003classifying} showed that if a group $G$ is torsion-free and hyperbolic relative to a finite collection of subgroups $\calP$, each of type-F, then $G$ is also of type $F$. 

\begin{definition}[{\cite[Definition 3.1]{manning2020cohomology}}]
\label{Def: Cylindrical cusped space}
    Let $(G,\calP)$ be a relative-hyperbolic pair of type $F$, and let $(\bar X,\bar\calA)$ be (finite) simplicial relative classifying space (cf. \Cref{notation: complexity}).
    We denote by $X$ the universal cover of $\bar X$ and by $A_i$ the lifts of $\bar \calA_i$ such that $(G,\calP) \actson (X,\{A_i\})$.
    The \emph{cylindrical cusped space} $\cyl{X,\calA}$ (or $\cyl{X}$ if $\calA$ is understood from the context) is the universal cover of the open mapping cylinder defined as: $$ \bar{X}\sqcup([0,\infty)\times\sqcup_{i=1}^n\bar A_i)\big/ {x\sim(0,x),\; \forall x\in \bar{A_i}}.$$ 
\end{definition}
 Following \cite{manning2020cohomology}, we equip $\cyl{X,\calA}$ with a $G$-equivariant simplicial structure whose vertex set is the vertices in $X$ together with all points in the lift of a vertex $(v,n)$ where $n\in\bbN_0$, for all vertices $v$ in $\bar A_i$ for all $1\le i\le n$.
    
The space $\cyl{X,\calA}$ used in \cite{manning2020cohomology} is manufactured specifically for the purpose of allowing arguments as in \Cref{prop: Cohomology  boundary isomorphism } below. We would like to utilize those results, together with the results on $\comb{X,\calA}$ described in \Cref{Pre: Mineyev's bicombing}.
Manning and Wang describe a metric on $\cyl{X,\calA}$ that gives rise to a quasi-isometry $\cyl{X,\calA}\to \comb{X,\calA}$.
This metric is the quotient metric defined as the Euclidean metric on simplices of $X$, and on each horoball $C = [0,\infty) \times A_i$ it is the \emph{warped product metric} $$C=[0,\infty)\times_{2^{-t}}A_i.$$
Here, the warped product is induced from the following length: 
Let $X$ and $Y$ be two length spaces and let $\varphi:[0,1]\to X\times Y$ be a path where $\varphi(t)=(\alpha(t),\beta(t))$.
The length of $\varphi$ is defined as the supremum over all partitions $\tau=\{0=t_0<\dots<t_n=1\}$ of $[0,1]$ of $$\calL_\tau(\varphi)=\sum_{i=1}^n \sqrt{d_X(\alpha(t_i),\alpha(t_{i-1}))^2+2^{-2\alpha(t_i)}d_Y(\beta(t_i),\beta(t_{i-1}))^2}.$$

 Notice that this metric on $\cyl{X,\calA}$ is not the simplicial shortest path metric, since elements in $t \times A_i$ are getting closer to each other as $t$ increases, but the simplicial shortest path between them never changes.   

\begin{proposition}[{\cite[Proposition 3.10]{manning2020cohomology}}]\label{Prop: Cusped spaces are quasi-isometric q.i} Given the above setting, $\cyl{X,\calA}$ is quasi-isometric to $\comb{X,\calA}$.
\end{proposition}

\begin{remark} \label{rem:q.i}
    \Cref{Prop: Cusped spaces are quasi-isometric q.i} gives an identification of $\partial(G,\calP) = \partial \comb{X,\calA}$ with the boundary of $\partial \cyl{X,\calA}$. 
\end{remark}

We define the depth of a vertex $(\bar v, t)\in\cyl{X}$ to be $\depth(v)=t$.
For a simplex $\sigma\in\cyl{X}$ we define $\depth(\sigma)$ to be the smallest interval containing $\depth(\sigma^0)$.
For $R\in\bbR$, we define $\cyl{X}_{\leq R}$ to be the complete subcomplex of $\cyl{X}$ spanned by vertices of depth at most $R$.     

In Section \ref{Section : construction of q.s map}, we would like to use the contractibility properties of a Rips-complex (\cite[III.$\Gamma$ Proposition 3.23]{bridson2013metric}).
We denote the Rips complex of parameter $D$ for a metric space $Z$ by $R_D(Z)$; this is the simplicial complex whose vertices are points of $Z$ and simplices are sets of diameter at most $D$.
we write $R_D(\cyl{X})$ to refer to the Rips complex on the vertices of $\cyl{X}$ with respect to the warped metric.
As vertices of $R=R_D(\cyl{X})$ can be identified with vertices of $\cyl{X}$, the depth of vertices is defined similarly.
We define $R_{\leq M}$ to be the full subcomplex on vertices at depth at most $M$.  

\begin{lemma}\label{lem: Homotopy equivalence}
    For all $D\geq 0$ there exists a proper continuous map $h:R_D(\cyl{X})\to \cyl{X}$, such that the restriction of $h$ to the $1$-skeleton $h_{|R_D(\cyl{X})^{(1)}}$ is a quasi-isometry.
\end{lemma}

\begin{proof}
    We repeat the proof of \cite[Lemma 3.13]{manning2020cohomology}, defining $h$ to map simplices in a depth preserving way. We then explain how we can amend $h$ to be a quasi-isometry on the edge set.

    Denote the Rips complex by $R:=R_D(\cyl{X})$.
    We say that a continuous map $h:R\to \cyl{X}$ is depth preserving if for each simplex $\sigma$ of $R$ and $I\subseteq [0,\infty)$ the smallest interval containing $\depth(\sigma^0)$, $$\depth(h(\sigma))\subseteq
\begin{cases}
[0,\sup I] & \inf I\le D \\
I &  \inf I>D
\end{cases}$$

     The vertices of $R$ are identified with the vertices of $\cyl{X}$. We proceed to define $h$ by induction on the dimension of simplices: 
     Let $\sigma\subset R$ be a representative for the $G$-orbits of $k$-simplices.
     If $\sigma$ is already in $\cyl{X}$, we define $h(\sigma)$ to be the identity.
     Otherwise, let $\depth(\sigma)=[a,b]$.
     Suppose first that $b>D$, then the vertices of $\sigma$ are contained in a certain horoball, and by the induction hypotheses $\depth(h(\partial\sigma))\subseteq [a,b]$. 
     Thus, $h(\partial\sigma)$ is contained in a certain horoball.
     Moreover, it is within a product of the form $[a,b]\times A_i$ for some $i$.
     As $A_i$ is contractible, so is the product, and we can continuously extend $h$ to $\sigma$ such that $\depth(h(\sigma))\subset [a,b]$. 
     Suppose now that $b\le D$, then $\depth(h(\partial\sigma))\leq D$, and so $h(\sigma)$ is contained in the subspace consisting of points at depth $\leq D$, which is again contractible. 
    
    In the last case, where $b\le D$, we can choose $h$ such that it maps edges in $R$ to paths of length at most $2^{2D}$.
    Otherwise, since horoballs are uniformly quasi-convex, we can choose $h$ such that it maps edges to paths in of length $D$, up to some uniform constants.
    That implies that $h$ is quasi-isometry on the $1$-skeleton. 
    
    Notice that $h$ is a $G$-equivariant, depth-preserving map.
    We can then conclude that it is proper (see \Cref{Phi is proper} for details).
\end{proof}

\subsection{Relative group cohomology}
Let $(G,\calP)$ be a group pair.
Bieri \cite{bieri1978relative} defines the cohomology group a pair in the following way: denote by $\Delta_{G/\calP}$ the kernel of the augmentation map $\epsilon: \oplus_i\bbZ G/P_i \to \bbZ $, defined by $\epsilon(xP_i)=1$.
Given a projective resolution of $\Delta_{G/\calP}$ as a $G$-module, namely an exact sequence of projective modules $$\dots \to F_n\to \dots \to F_1\to\
F_0\to \Delta_{G/\calP}$$
one can apply the covariant functor $Hom(-,M)$, then define:
\begin{definition}\label{def:rel_cohm}
    The \emph{(relative) cohomology of the pair} $(G,\calP)$ with coefficients in M is defined by $$H^k(G,\calP;M):=H^{k-1}(Hom(F,M))=Ext_{\bbZ G}^{k-1}(\Delta_{G/\calP}; M)$$
\end{definition}

\begin{definition}
    The \emph{cohomological dimension} of $(G,\calP)$ is $$\cd(G,\calP)=\max\{n\in\bbN | \exists \ G\text{-module }M \text{ with } H^n(G,\calP;M)\neq0\}.$$ 
\end{definition}

We will consider relatively hyperbolic torsion-free pairs $(G,\calP)$.
We further assume that $\calP$ is a finite collection of type $F$ subgroups, which implies that $(G,\calP)$ is of type $F$ (\cite{dahmani2003classifying}). 

\begin{proposition} [{\cite[Lemma 2.9]{kapovich2009homological}}]
    If $(G,\calP)$ is of type $F$, then  $$\cd(G,\calP)=max\{n\in\bbN \;|\; H^n(G,\calP;\bbZ G)\neq0\}$$
\end{proposition}

\begin{lemma}\label{lem:coh.dim}
Let $(G,\calP)$ be a type-F group pair with $\calP\neq\{G\}$ and $G$ non-trivial.
Then there exists $k\ge1$ for which the relative cohomology group $H^k(G,\calP;\bbZ G)$ is non-trivial.  
\end{lemma}
\begin{proof}
Denote by $\prd(N)$ the projective dimension of a $G$-module $N$.
By \cite[VIII,2.1]{brown2012cohomology}, $\cd(G, \calP)-1=\prd(\Delta_{G/\calP})$. 

Notice that $\prd(N)=-1$ if and only if there exists a resolution $0\to N \to 0$, which implies $N=0$. In particular, for $N=\Delta_{G/\calP}$, this implies that $\epsilon$ is injective.

It remains to note that if $\oplus_i\bbZ G/P_i \to Z$ is injective, $\calP$ cannot contain more than one element $P$, thus the quotient $G/P$ must be trivial, i.e. $G$ must be equal to $P$.
\end{proof}

Finally, the Bowditch boundary of a relatively hyperbolic group entirely encodes the relative group cohomology in the following sense:

\begin{theorem}[{\cite[Theorem 1.1]{manning2020cohomology}}]
\label{prop: Cohomology  boundary isomorphism }
If $(G,\calP)$ is relatively hyperbolic pair of type $F_\infty$, then for every $k$ there is an isomorphism of $AG$-modules $$H^k(G,\calP;AG)\to\checkH{k-1}{\partial(G,\calP);A}$$
\end{theorem}

\section{Uniform quasi-surjectivity}\label{Section: Quasi surjectivity}

The goal of this section is to prove the following theorem on uniform coarse surjectivity for cyclindrical cusped spaces of finite-index subgroups.

\begin{theorem}\label{Theorem: Quasi-suerjectivity} 
Let $(G,\calP)$ be a torsion-free relatively hyperbolic group, and let $(\bar X,\bar \calA)$ be a relative classifying space for $(G,\calP)$.
There exists $R_0 = R_0(X)$ such that for every subgroup $H\leqslant G$ of finite index, every relative classifying space $(\bar K,\bar \calB)$ for $(H,\calP_H)$, and every continuous $H$-equivariant map $\Phi: \cyl{K,\calB} \to \cyl{X,\calA}$ that extends continuously to a homeomorphism $\partial\Phi : \partial\cyl{K,\calB} \to\partial \cyl{X,\calA}$, we have $ X\subseteq N_{R_0}(\Phi(\cyl{K,\calB}))$.
\end{theorem}

\begin{proof}
We denote  $\cusped{Z}=\cyl{Z,\calD}$, whenever it is clear from context which is the relevant space. By \Cref{rem:q.i}, the Bowditch boundary $\partial(G,\calP)$ is homeomorphic to $\partial\cusped{X}$, thus, by the assumption on $\Phi$, it extends to a boundary map $\partial(H,\calP_{H}) \to \partial(G,\calP)$.

 By \Cref{lem:coh.dim} and \Cref{prop: Cohomology  boundary isomorphism } there is some $k\in\mathbb{N}$ for which 
 \begin{equation}\label{equ: dim_non_zero}\checkH{k-1}{\partial(G,\calP)}\cong H^k(G,\calP;\bbZ G)\neq 0
 \end{equation}

Moreover, by \cite[Proposition 3.26]{manning2020cohomology}, there is an isomorphism 
$$\checkH{k-1}{\partial(G,\calP)}\xrightarrow{i} H^k_c(\cusped{X})$$
achieved through a long exact sequence in cohomology.

Similarly for $K$, there is an isomorphism:
$$\checkH{k-1}{\partial(H,\calP_H)}\to H^k_c(\cusped{K})$$

Since $G$ acts on $X$ cocompactly, it is enough to show that given an arbitrary $x\in X$, there exists some $R$ such that $\Phi(\cusped{K})\cap B_{R}(x)$ is non-empty, where $B_{R}(x)=B_{R}$ is an open ball based at $x$ of radius $R$.
By the naturality of the map $i$ in the long exact sequence we have the following commutative diagram:
\begin{center}
    
\begin{tikzcd}
  \checkH{k-1}{\partial(G,\calP)} \arrow[r,hook, "i"] \arrow[d, "(\partial(\Phi))^*"]
    & H^k_c(\cusped{X}) \arrow[d, "\Phi^*"] \\
 \checkH{k-1}{\partial(H,\calP_H)} \arrow[r, hook,"i"]
& H^k_c(\cusped{K})\end{tikzcd}\end{center}
 By assumption (\Cref{equ: dim_non_zero}), there exist some 
 $0\neq a\in \checkH{k-1}{\partial\cusped{X}}$. Since $i$ is injective, the image $i(a)$ is compactly supported in $\cusped{X}$, i.e. there exists some $R_0>0$ such that $0\ne i(a)\in H^k(\cusped{X},\cusped{X}-B_{R_0})$.
 
It remains to notice that $\Phi(\cusped{K})\cap B_{R_0}\neq \emptyset$. Otherwise, we would get that $\Phi^*(i(a))=0$, which leads to a contradiction: since $a\neq 0$, and $(\partial\Phi)^*$ is an isomorphism, $a':=(\partial\Phi)^*(a)$ is also nontrivial, and so $i(a')\neq 0$, in contradiction with the commutativity of the diagram. 
\end{proof}

\begin{definition}[depth D-non-decreasing map]
Let $\Phi$ be a map $\Phi:\cyl{Y}\to\cyl{Z}$ for some $Y$, $Z$.
We say that $\Phi$ is \emph{depth $D$-non-decreasing} if there exist some constant $D$ such that for every simplex $\sigma$ in $\cyl{Y}$, $\depth(\Phi(\sigma))\ge \depth(\sigma)-D$.
\end{definition}

For the remainder of this section we fix  $(K,\calB)$ and $(X,\calA$) as in \Cref{Theorem: Quasi-suerjectivity}. 
\begin{lemma}\label{phi extends to boundary}
    If $\Phi: \cyl{K,\calB} \to \cyl{X,\calA}$ is $H$-equivariant, depth $D$-non-decreasing, and $\Phi$ restricts to a quasi-isometry on vertices, then $\Phi$ extends continuously to a map $\partial \Phi$.
\end{lemma}

\begin{proof} We will need the following claims:
    \begin{claim}\label{Phi is proper}
    $\Phi$ is proper. 
    \end{claim}
    \begin{proof}
        Let $\set{k_i}$ be a sequence in $\cyl{K}$ escaping every compact subset. Assume by contradiction that $\{\Phi(k_i)\}$ does not escape every compact subset. In particular, one can find some compact set in $\cyl{X}$ containing a subsequence $(\Phi(k_{i_j})).$ 
    
        Thus, there is some $M\in \bbR$ such that  $\depth(k_{i_j})\le M$ for all $i$. %Since $H$ acts cocompactly on $\cyl{K}_{\le M}$, 
        As $\Phi$ is depth $D$-non-decreasing, it must be that also $\depth k_{i_j}$ is bounded. Since $H$ acts cocompactly on $\cyl{K}_{\le M}$ and $\{k_{i_j}\}$ escapes every compact set, we can find sequences $\set{g_i}\subseteq H$ and $\set{k'_{i_j}}\subseteq K$ with $g_{i_j}$ escaping every compact subset of $H$, and $\{k_{i_j}'\}$ in some compact subset, such that $k_{i_j}=g_{i_j}.k_{i_j}'$. As $\Phi$ is $H$-equivariant, $\Phi(k_{i_j})=\Phi(g_{i_j}.k_{i_j}')=g_{i_j}.\Phi(k_{i_j}')$, which also escapes every compact subset. This is a contradiction.
    \end{proof}
    
    \begin{claim}[{\cite[Lemma 3.14]{manning2020cohomology}}] \label{limits behaves nicely in cyl{X}}
    Let $\{x_n\}, \{y_n\}$ be sequences in $\cyl{X,\calA}$ tending to infinity. Assume the following:
    \begin{enumerate}
        \item $x_n \to \xi , y_n \to \xi'$ for $\xi,\xi' \in \partial\cyl{X,\calA}$.
        \item For each $n$, $x_n$ and $y_n$ are in the same horoball and $$\min\{\depth(y_n)\, \depth(x_n)\}\to \infty.$$
    \end{enumerate}
    Then $\xi = \xi'$.
    \end{claim}

Let $\{k_i\}, \{k_i'\} \subseteq \cusped{K}$ be two sequences converging to the same boundary point $l \in \partial\cusped{K}$.
Passing to subsequences, we may assume that $\Phi(k_i) \to \xi$ and $\Phi(k_i') \to \xi'$ for some $\xi, \xi' \in \partial\cusped{X}$. 

Note that since $\cusped{K}$ and $\cusped{X}$ are respectively dense in $\cusped{K}\cup\partial\cusped{K}$ and $\cusped{X}\cup\partial\cusped{X}$, to show that $\Phi$ extends continuously to a map $\partial\Phi$, it suffices to show that $\Phi(k_i)$ and $\Phi(k_i')$ converge to the same point in $\cusped{X}$.
This would imply that defining $\partial\Phi(l) := \lim \Phi(k_i)$ yields a well-defined map, and that $\Phi$ extends continuously to the boundary.

%By \Cref{Phi is proper}, 
Suppose $\Phi(k_i)\to\xi$. 
For each j, let $v_j$ be a vertex in a simplex $\sigma_j$ containing $k_j$, and consider the sequence $\Phi(v_i)$. Assume first that $\depth (k_i)$ is bounded by some $M\in\bbR$.
Since $\cusped{K}$ is finite dimensional, the diameter of simplices in $\cusped{K}$ is uniformly bounded. The action $H\actson \cusped{K}_{\leq M}$ is cocompact and $\Phi$ is $H$-equivariant, so also $d(\Phi(v_i),\Phi(k_i))$ is bounded. We conclude that $\Phi(v_i)$ tends to $\xi$.
Since $\Phi$ is a quasi-isometry on vertices, $\Phi$ can be extended to a boundary map when restricted to the $0$-skeleton, then by repeating the process for $k_i'$, we conclude that $\Phi(k_i), \Phi(k_i')$ have the same limit point in $\cusped{X}$.

Suppose now that $\Phi(k_i)$ has unbounded depth. 
Then, up to passing to subsequences, we also have $\depth(k_i)\to \infty$, and thus also $\depth(v_i)\to\infty$.
So, also the simplex $\sigma_i$ whose depth is $\depth(\sigma_i)=[a_i,b_i]$, has $a_i\to \infty$.
In particular, for all but finitely many $i$, $v_i$ and $k_i$ are in the same horoball.
Since $\Phi$ is depth $D$-non decreasing, it must be that for all but finitely many $i$, $\Phi(v_i)$ and $\Phi(k_i)$ are in the same horoball (as otherwise, the images of simplices of $\cusped{K}$ would have vertices at two distinct horoballs and at unbounded depth, and that would mean that $\Phi$ cannot be depth $D$-non decreasing). 
Thus, by \Cref{limits behaves nicely in cyl{X}}, they have the same limit point. 
\end{proof}

\section{Construction of quasi-surjective map}\label{Section : construction of q.s map}
Let $G$ be hyperbolic relative to the collection $\calP$ of finitely generated subgroups of type-F, and let $H\leqslant G$ be a subgroup of finite index with classifying spaces $( \bar K, \bar \calB)$ and $(\bar X, \bar \calA)$ for $(H,\calP_H)$ and $(G,\calP)$ respectively.

In this section, we construct a map $\Phi: \cyl{K}\to\cyl{X}$ on which we can apply \Cref{Theorem: Quasi-suerjectivity}.
Since $\cyl{X,\calA}$ is hyperbolic, following \cite{bridson2013metric}, we can construct a corresponding contractible Rips-complex $Y:=R_D(\cyl{X})$ (see \S\ref{Subsec: cylindrical cusped space}).

\begin{theorem}\label{quasisurjectivity upgraded}
    There exist $\kappa, R$ that depend only on $X$, such that for every $H\leqslant G$ of finite index we have that every $H$-equivariant map $\Phi^0: K^0 \to X^0$, $\Phi^0$ extends to $\Phi:\cyl{K} \to \cyl{X}$ such that:
    \begin{enumerate}
        \item $\Phi$ sends edges to $\kappa$-quasigeodesics.
        \item $X \subseteq N_R(\Phi(K^1))$.
    \end{enumerate}
\end{theorem}

\begin{proof}
The outline is as follows. We start with a map $\cyl{K}\to Y$, the image of which is $\kappa$-close to geodesics (using \cite[Lemma 2.3]{manning2020cohomology}), we then project its image to $\cyl{X}$ and apply \Cref{Theorem: Quasi-suerjectivity}. Finally, we show that $X$ can be found in a neighborhood of the image of $\Phi$ when restricted to the $1$-skeleton $K^{(1)}$. 

 \textbf{Step 1.} We construct $\tild \Phi : \cyl{K} \to Y$, by extending $\Phi^0$ to a $H$-equivariant continuous map $\tild{\Phi}:\cyl{K} \to Y$.
 
 By induction on the dimension of simplices: first, we define $\tild \Phi$ on vertices of $\cyl{K}$ to be $\tild{\Phi}((v,t))=(\Phi(v),t)$.
 For every edge $e\in \cyl{K}^1$ define $\tild{\Phi}|e$ to be a geodesic in $Y$ connecting $\Phi(e_-)$ and $\Phi(e_+)$ in a $H$-equivariant way.
 Assume $\tild{\Phi}$ was defined on $K^{(n-1)}$ for $n>1$. Since $Y$ is contractible we can define $\tild{\Phi}|_{K^{(n)}}$ for an $n$-simplex $\sigma$ to be a filling disk of the $(n-1)$-simplex $\tild{\Phi}|_{\partial\sigma}$ in $Y$, then extend $H$-equivariantly. Since horoballs are uniformly quasi-convex, we can choose this filling disk in a depth $D$-non-decreasing way.
 Moreover, since $Y$ is a $\boldsymbol\delta$-hyperbolic space, and since $\cyl{K}$ is finite dimensional simplicial complex, it follows from the proof of \cite[III.$\Gamma$ proposition 3.23]{bridson2013metric} that this filling disk can be chosen in some $\boldsymbol \delta$-neighborhood of $\tild{\Phi}(\cyl{K}^{(1)})$ (see \cite[Lemma 2.3]{manning2020cohomology}).
 So we can assume:

$$\tild{\Phi}(\cyl{K})\subseteq\calN_{\boldsymbol{{\delta}}}(\tild{\Phi}(\cyl{K}^{(1)}))$$

To sum-up this step, we have constructed a continuous $H$-equivariant depth $D$-non-decreasing map $\tild\Phi$, such that for all $t\in \bbN$ and for all $v\in K^0$, $\tild \Phi(v,t) = (\tild\Phi(v),t)$, $\tild \Phi$ sends edges to geodesics and $\tild \Phi(\cyl{K})$ is in a $\boldsymbol{\delta}$-neighborhood of the image of the $1$-skeleton $\tild \Phi(\cyl{K}^{(1)})$.

\textbf{Step 2.} We construct a continuous $H$-equivariant map $\Phi: \cyl{K} \to \cyl{X}$ by composing $\Phi = h \circ \tild \Phi$, where $h:Y\to\cyl{X}$ is the map constructed in \Cref{lem: Homotopy equivalence}.

Notice first that $h$ is continuous depth $D$-non-decreasing, $H$-equivariant, and so the composition $h \circ \tild \Phi$ also exhibits these properties.
Moreover, $h$ restricts to the identity on vertices, and maps edges to quasigeodesics.
Therefore, the composition $h \circ \tild \Phi$ sends edges to $\kappa$-quasigeodesics, and for all $t\in \bbN$ and for all $v\in K^0$, $\Phi(v,t) = (\Phi(v),t)$. 

Let $R_0\ge 0$. The preimage of $\Phi(\cyl{K})_{\leq R_0}$ by $h$ is at some bounded depth $R_0'$ in $Y$ (otherwise, $h$ could not have been depth $D$-non-decreasing).
As $G$ acts cocompactly on $Y_{\leq R_0'}$, and $h$ is $G$-equivariant, the filling of $k$-simplices in $Y_{\leq R_0'}$, is mapped by $h$ to $\Phi(\cyl{K})_{\leq R_0}$ and is controlled by some constant depending only on $X$ and the depth $R_0$.
We conclude that $$\Phi(\cyl{K})_{\leq R_0} \subseteq N_{\boldsymbol{\delta}_{R_0}}(  \Phi(\cyl{K}^{(1)})).$$
  
\textbf{Step 3.} We show that there exists $R_0=R_0(X)$ such that $$X\subseteq  N_{R_0}(\Phi(\cyl{K}^{(1)})):$$

By \Cref{phi extends to boundary}, $\Phi$ extends continuously to a map $\partial \Phi$. 
It now follows from \Cref{Theorem: Quasi-suerjectivity} that there exists some $R_0=R_0(X)$ such that $X\subseteq N_{R_0}(\Phi(\cyl{K}))$.
The desired conclusion then follows from the previous step by noticing that also $X\subseteq N_{R_0}(\Phi(\cyl{K})_{\leq R_0})$.

\textbf{Step 4.} Finally, we show that there exist $R=R(X)$, such that $$X\subseteq N_{R}(\Phi(K^{(1)})).$$

     Let $x\in X$. 
     % By \Cref{Theorem: Quasi-suerjectivity}, there exists some $v\in \Phi(\cyl{K})$ such that $x\in B_R(v)$. 
     By Step 3, there exists some $v\in \Phi(\cyl{K}^{(1)})$ such that $x\in B_{R_0}(v)$.
     If $v=\Phi((w,t))$ for some $w\in \cyl{K^0}$, then by the construction of $\Phi$,  $v$ is in the vertex set of $\cyl{X}$, and the projection of $v$ onto $X$, namely $\Phi(w)$, is also in $B_{R_0}(v)$. 
     Next, if $v$ is on the image of a vertical edge then $v$ is on a vertical geodesic, and the vertical projection of $v$ onto $X$ is also in $\Phi(\cyl{K})$. 
     Assume now that $v$ is on the image of a horizontal edge, i.e. on some quasigeodesic. By the Morse lemma, each such quasigeodesic remains within a uniformly bounded distance $M$ of a geodesic $\gamma_t:=\geod{\Phi(w,t),\Phi(w',t)}$. 
     Again, by construction, it must be that the vertical projections onto $X$ of $w_t:=\Phi(w,t)$, $w'_t:=\Phi(w',t)$ are $w_0:=\Phi(w)$ and $w'_0:=\Phi(w')$ respectively. 
     We want to find some $v'$ on the geodesic $\gamma_0=\geod{w_0,w'_0}$ that is also in $B_R$ for some $R=R(X)$. 
     Consider the $2\delta$-thin rectangle formed by $w_t, w'_t, w_0, w'_0$. 
     We know that $v$ is in a $2\delta$ neighborhood of either $\gamma_0$, $\geod{w_0, w_t}$ or $\geod{w'_0,w'_t}$. 
     Thus, it is at a distance $\leq R_0+2\delta$ from either $w_0$, $w'_0$ or at a distance $2\delta$ from some element in $\gamma_0$. 
     At any case, setting $R:=R_0+2\delta$+$2M$, we have shown that \begin{equation}\label{equ:q.s}
        X\subseteq  N_R(\Phi(\cyl{K})^{(1)})\implies X\subseteq N_{R}(\Phi(K^{(1)})). 
     \end{equation} 
     and we have shown in the previous step that $\Phi$ satisfies the antecedent of \Cref{equ:q.s}.
\end{proof}

\section{Volume vs Complexity}\label{Section: proof of main}
\subsection{Proof of \Cref{Theorem: Main}}
Let $(\bar{X},\bar{\calA})$ be a relative classifying space for the pair $(G,\calP)$, and $(X,\calA)$ the respective universal cover.
To prove \Cref{Theorem: Main}, we will show the following: $$ \frac{1}{\bdelta}[G:H]\leq \bfw(\bar{\calF}_{\le \bdelta})\leq \boldsymbol\delta\cdot\C(H,\calP_H)$$

By \cite{dahmani2003classifying} we know that if a group $G$ is torsion-free and hyperbolic relative to a finite collection $\calP$ of type-F subgroups, then $G$ is of type-F and $\C(G,\calP)$ is finite.
 
\begin{proof}
Let $K$ be an aspherical simplicial complex, such that $\vol{K/H} =\C(H,\calP_H)$ and let $L=K^{(2)}$.
By \Cref{prop;upper bound}, there exist a $G$-equivariant map $\Phi_0:L^0\to X^0$ and an associated resolution $(\calF_{\le R_0},\,\bfw)$, such that for every $R_0$ there exists a constant $\delta = \delta(X,R_0)$ such that \begin{equation}\label{equ:W<C}
\bfw (\bar{\calF}_{\le R_0})\leq \delta\cdot \vol{\bar{L}}\leq\delta\cdot\C(H,\calP_H)
\end{equation}

By \Cref{quasisurjectivity upgraded}, we can extend $\Phi_0$ to a $H$-equivariant map $\Phi:\cyl{K}\to\cyl{X}$ such that there exists some $R=R(X)$ with 

\begin{equation}\label{equ:d-neigh}
    X\subseteq N_R(\Phi(K^{(1)}))
\end{equation}

By \Cref{Prop: Cusped spaces are quasi-isometric q.i} there exists a quasi-isometry $\iota:\cyl{X}\to\comb{X}$.
Denote by $\Upsilon$ the map $\iota\circ\Phi$. 

Fix a point $x\in X$. By \Cref{equ:d-neigh}, there exists some edge $e$ in $K$ such that $x$ lies within distance $R$ of $\Phi(e)$.
Recall that $\Phi$ maps edges to $\kappa$-quasigeodesics with $\kappa=\kappa(X)$. Since $\iota$ is quasi-isometry, the composition $\iota\circ\Phi$ sends edges to uniform quasi-geodesics with constant depending only on $X$.
By the Morse lemma, each such quasi-geodesic remains within a uniformly bounded distance of a geodesic in $\comb{X}$. 
So, we can assume that  $R$ was chosen to be big enough such that there is an $R$-neighborhoood of $x$, in which we can find some point $z$ contained in a geodesic $[\Upsilon(e_-),\Upsilon(e_+)]$ such that $\Upsilon(e_-)\neq \Upsilon(e_+)$ for some $e\in K$. Moreover, by \Cref{cor : <1/lambda} there exist $\rho=\rho(X)$ and $\lambda=\lambda(X,R)$,  and an edge $f$ in $B_\rho(z)$ whose coefficient in $q(e)$ is at least $1/\lambda$.
Let $R_0 = R+\rho$.
Since $R, \rho$ depend only on $X$, so do $R_0,\;\lambda$ and $\delta$ (of \eqref{equ:W<C}).
The above discussion shows that in $B_{R_0}(x)$ there exists an edge $f$ whose coefficient in some $q(e)$ is at least $1/\lambda$.

There are at least $[G:H]$ vertices in $X/H$. 
Denote by $\alpha$ the number of vertices in a ball in $\comb{X}$ of radius $2R_0$ around a vertex $x\in X$.
Then, there are at least $\frac{1}{\alpha}[G:H]$ disjoint ${R_0}$-balls in $X/H$. 
Each such ball gives rise to (at least) one track with weight of at least $1/\lambda$ in $\bar \calF_{\le R_0}$.
Therefore, $$\frac{1}{\lambda \alpha}[G:H] \le \bfw(\bar \calF_{\le R_0}) \le \delta \C(H,\calP_H).$$ 
This completes the proof of the theorem.
\end{proof}

\subsection{Finite index rigidity}
Finally, we provide complete proofs for \Cref{thm:fin-ind}, \Cref{thm:nrh}, and \Cref{cor:list}.

\begin{namedtheorem}[Theorem \ref{thm:fin-ind}]
Let a group $G\not\simeq \bbZ$ be torsion-free, hyperbolic relative to type-F proper subgroups $\calP$.
    Then, if $H,H'$ are finite index subgroup of $G$ such that $(H,\calP_H)\simeq (H',\calP_{H'})$ then $[G:H]=[G:H']$.
\end{namedtheorem}

\begin{proof}
If $G$ is multi-ended then it is finite index rigid by \cite{sykiotis2018complexity}.
We may thus assume that $G$ is one-ended.
Following \cite{reznikov1995volumes}, we can define: $$\underline{\C}(G,\calP):= \liminf_{H\leq G, [G:H]<\infty} \frac{\C(H,\calP_H)}{[G:H]}.$$
This group pair invariant is multiplicative in the following sense:
$$\underline{\C}(H,\calP_H) = [G:H]\cdot\underline{\C}(G,\calP).$$
If $H,H'\le G$ are finite index subgroups of $G$ such that $(H,\calP_H) \simeq (H',\calP_{H'})$ then
 \begin{equation}\label{long equality} [G:H]\cdot\underline{\C}(G,\calP)=\underline{\C}(H,\calP_H)=\underline{\C}(H',\calP_{H'})=[G:H']\cdot\underline{\C}(G,\calP)
 \end{equation}
By \Cref{Theorem: Main} and the inequality \eqref{easy inequality}, $0<\alpha\le \underline{\C}(G,\calP) \le \beta <\infty$. Hence,  $[G:H]=[G:H']$.
\end{proof}

% \begin{namedtheorem}[\Cref{Thm: nice result}]
%   Let $G\not\cong\bbZ^n$ be torsion-free toral relative hyperbolic group, then $G$ is finite index rigid.
% \end{namedtheorem}

% \begin{proof}
% Since $G$ is torsion-free hyperbolic relative to finitely generated abelian subgroups, $G$ is also hyperbolic relative to the collection $\calP_{G,ab}$ of maximal free abelian subgroups of rank $\ge 2$ in $G$ (up to conjugation).
% Observe that:
% \begin{enumerate}
%     \item \label{Pab for subgroup}If $H$ is a finite index subgroup of $G$, then $\calP_H = \calP_{H,ab}$.
%     \item \label{Pab under isomorphism} If $\phi:G\to G'$ is an isomorphism, then $\phi(\calP_{G,ab})=\calP_{G',ab}$.
% \end{enumerate}
% Then, if $H\simeq H'$ for finite index subgroups, then $(H,\calP_H)\simeq (H',\calP_{H'})$, and the result follows from \Cref{thm:fin-ind}.
% \end{proof}

Recall that we say that an infinite finitely generated group $G$ is \emph{NRH} if it is not hyperbolic relative to any collection of proper subgroups.

\begin{namedtheorem}[\Cref{thm:nrh}]
  Let a group $G\not\simeq\bbZ$ be torsion-free, hyperbolic relative to type-F and NRH proper subgroups $\calP$. Then $G$ is finite index rigid.
\end{namedtheorem}

\begin{proof}
Because of \Cref{thm:fin-ind}, it suffices to argue that an isomorphism between finite-index subgroups $H,H'$ of $G$ is necessarily an isomorphism between the pairs $(H,\calP_H)$ and $(H',\calP_{H'})$. The elements of $\calP_H$ are (isomorphic to) finite-index subgroups of elements of $\calP$, and therefore they are also NRH. This is because being NRH is even a quasi-isometry invariant \cite{drutu2009relatively}. Fix now any isomorphism $\phi:H\to H'$. Consider also any $P\in\calP_{H}$. Note that $P$ is undistorted, since it is quasiconvex by \cite[Lemma 5.4]{osin2006relatively}, and hence so is $\phi(P)$. Therefore, since $\phi(P)$ is NRH, \cite[Corollary 4.7]{behrstock2009thick} implies that $\phi(P)$ is contained in a conjugate $P'$ of some element $P'$ of $\calP_{H'}$. Applying the same argument to $\phi^{-1}$ and $P'$, we see that $P=\phi^{-1}(\phi(P))<\phi^{-1}(P')<P''$ for some conjugate $P''$ of some element of $\calP_H$. This forces $P=P''$, and hence all containments to be equalities, because distinct conjugates of infinite peripheral subgroups cannot be strictly contained in each other, see \cite[Proposition 2.36]{osin2006relatively}.

To sum up, for all $P\in\calP_H$ there exists $\pi(P)\in \calP_{H'}$ such that $\phi(P)$ is conjugate to $\pi(P)$. Note that $\pi:\calP_H\to\calP_{H'}$ needs to be a bijection, with inverse constructed in the same way except starting with $\phi^{-1}$. This concludes the proof.
\end{proof}

\begin{namedtheorem}[\Cref{cor:list}]
The following are finite-index rigid:
\begin{enumerate}
\item fundamental groups of complete finite-volume manifolds of pinched negative curvature,
\item non-abelian limit groups,
    \item torsion-free group hyperbolic relative to nilpotent subgroups,
    \item free-by-cyclic groups with respect to an exponentially growing automorphism.
\end{enumerate}
\end{namedtheorem}

\begin{proof}
    Items (1) and (2) are special cases of (3), see respectively \cite{farb1998relatively} and \cite{dahmani2003combination}. For $G$ as in (3), we have that $G$ is also hyperbolic relative to a collection of non-virtually-cyclic nilpotent groups, because virtually cyclic subgroups can always be removed from the collection of peripheral subgroups preserving relative hyperbolicity, see \cite[Corollary 1.14]{dructu2005tree}, and non-virtually-cyclic nilpotent groups are NRH, see \cite[Corollary 6.14]{dructu2005tree}.

    Finally, (4) follows from \cite[Corollary 3.16]{ghosh2023relative} (see also \cite[Theorem 4]{dahmani2022relative}), which says that free-by-cyclic groups are hyperbolic relative to free-by-cyclic subgroups with respect to polynomially growing automorphisms, combined with \cite[Theorem 1.2]{hagen2019remark} (see also \cite{macura2002detour}), which says that these are thick in the sense of \cite{behrstock2009thick}, whence NRH by \cite[Corollary 7.9]{behrstock2009thick}.
\end{proof}

% \begin{corollary}
%   Let $G\not\cong\bbZ^n$ be torsion-free toral relative hyperbolic group, then $G$ is finite index rigid.
% \end{corollary}

% \begin{proof}
%     Since $G$ is torsion-free hyperbolic relative to finitely generated abelian subgroups, $G$ is also hyperbolic relative to the collection $\calP_{G,ab}$ of maximal free abelian subgroups of rank $\ge 2$ in $G$ (up to conjugation), and we can apply \Cref{thm:nrh}.
% \end{proof}

\bibliographystyle{plain}
\bibliography{biblio}

\begin{thebibliography}{10}

\bibitem{behrstock2009thick}
Jason Behrstock, Cornelia Dru{\c{t}}u, and Lee Mosher.
\newblock Thick metric spaces, relative hyperbolicity, and quasi-isometric rigidity.
\newblock {\em Math. Ann.}, 344(3):543--595, 2009.

\bibitem{bestvina1991boundary}
Mladen Bestvina and Geoffrey Mess.
\newblock The boundary of negatively curved groups.
\newblock {\em Journal of the American Mathematical Society}, 4(3):469--481, 1991.

\bibitem{bieri1978relative}
Robert Bieri and Beno Eckmann.
\newblock Relative homology and {P}oincar{\'e} duality for group pairs.
\newblock {\em Journal of Pure and Applied Algebra}, 13(3):277--319, 1978.

\bibitem{bowditch2012relatively}
Brian~H Bowditch.
\newblock Relatively hyperbolic groups.
\newblock {\em International Journal of Algebra and Computation}, 22(03):1250016, 2012.

\bibitem{bridson2013metric}
Martin~R Bridson and Andr{\'e} Haefliger.
\newblock {\em Metric spaces of non-positive curvature}, volume 319.
\newblock Springer Science \& Business Media, 2013.

\bibitem{brown2012cohomology}
Kenneth~S Brown.
\newblock {\em Cohomology of groups}, volume~87.
\newblock Springer Science \& Business Media, 2012.

\bibitem{dahmani2003classifying}
Fran{\c{c}}ois Dahmani.
\newblock Classifying spaces and boundaries for relatively hyperbolic groups.
\newblock {\em Proceedings of the London Mathematical Society}, 86(3):666--684, 2003.

\bibitem{dahmani2003combination}
Fran{\c{c}}ois Dahmani.
\newblock Combination of convergence groups.
\newblock {\em Geometry \& Topology}, 7(2):933--963, 2003.

\bibitem{dahmani2022relative}
Fran{\c{c}}ois Dahmani and Ruoyu Li.
\newblock Relative hyperbolicity for automorphisms of free products and free groups.
\newblock {\em Journal of Topology and Analysis}, 14(01):55--92, 2022.

\bibitem{drutu2009relatively}
Cornelia Dru{\c{t}}u.
\newblock Relatively hyperbolic groups: geometry and quasi-isometric invariance.
\newblock {\em Comment. Math. Helv.}, 84(3):503--546, 2009.

\bibitem{dructu2005tree}
Cornelia Dru{\c{t}}u and Mark Sapir.
\newblock Tree-graded spaces and asymptotic cones of groups.
\newblock {\em Topology}, 44(5):959--1058, 2005.

\bibitem{farb1998relatively}
Benson Farb.
\newblock Relatively hyperbolic groups.
\newblock {\em Geometric and functional analysis}, 8(5):810--840, 1998.

\bibitem{ghosh2023relative}
Pritam Ghosh.
\newblock Relative hyperbolicity of free-by-cyclic extensions.
\newblock {\em Compositio Mathematica}, 159(1):153--183, 2023.

\bibitem{gromov1987hyperbolic}
Mikhael Gromov.
\newblock Hyperbolic groups.
\newblock In {\em Essays in group theory}, pages 75--263. Springer, 1987.

\bibitem{groves2008dehn}
Daniel Groves and Jason~Fox Manning.
\newblock Dehn filling in relatively hyperbolic groups.
\newblock {\em Israel Journal of Mathematics}, 168:317--429, 2008.

\bibitem{hagen2019remark}
Mark Hagen.
\newblock A remark on thickness of free-by-cyclic groups.
\newblock {\em Illinois J. Math.}, 63(4):633--643, 2019.

\bibitem{kapovich2009homological}
Michael Kapovich.
\newblock Homological dimension and critical exponent of {K}leinian groups.
\newblock {\em Geometric and Functional Analysis}, 18(6):2017--2054, 2009.

\bibitem{lazarovich2023finite}
Nir Lazarovich.
\newblock Finite index rigidity of hyperbolic groups.
\newblock {\em arXiv preprint arXiv:2302.04484}, 2023.

\bibitem{macura2002detour}
Nata{\v{s}}a Macura.
\newblock Detour functions and quasi-isometries.
\newblock {\em The Quarterly Journal of Mathematics}, 53(2):207--239, 2002.

\bibitem{manning2020cohomology}
Jason~F Manning and Oliver~H Wang.
\newblock Cohomology and the {B}owditch boundary.
\newblock {\em Michigan Mathematical Journal}, 69(3):633--669, 2020.

\bibitem{mineyev2001straightening}
Igor Mineyev.
\newblock Straightening and bounded cohomology of hyperbolic groups.
\newblock {\em Geometric \& Functional Analysis GAFA}, 11(4):807--839, 2001.

\bibitem{mostow1968quasi}
George~D Mostow.
\newblock Quasi-conformal mappings in $ n $-space and the rigidity of hyperbolic space forms.
\newblock {\em Publications Math{\'e}matiques de l'IH{\'E}S}, 34:53--104, 1968.

\bibitem{mostow1973strong}
Grigory~A Mostow.
\newblock {\em Strong rigidity of locally symmetric spaces}.
\newblock Number~78. Princeton University Press, 1973.

\bibitem{osin2006relatively}
Denis~V Osin.
\newblock {\em Relatively Hyperbolic Groups: Intrinsic Geometry, Algebraic Properties, and Algorithmic Problems: Intrinsic Geometry, Algebraic Properties, and Algorithmic Problems}, volume 843.
\newblock American Mathematical Soc., 2006.

\bibitem{prasad1973strong}
Gopal Prasad.
\newblock Strong rigidity of {Q}-rank 1 lattices.
\newblock {\em Inventiones mathematicae}, 21(4):255--286, 1973.

\bibitem{reznikov1995volumes}
Alexander Reznikov.
\newblock Volumes of discrete groups and topological complexity of homology spheres.
\newblock {\em arXiv preprint dg-ga/9506010}, 1995.

\bibitem{sykiotis2018complexity}
Mihalis Sykiotis.
\newblock Complexity volumes of splittable groups.
\newblock {\em Journal of Algebra}, 503:409--432, 2018.

\end{thebibliography}

\Addresses

\end{document}